\documentclass{amsart}
\usepackage{amscd,amsmath,amssymb,amsfonts}
\usepackage[T1]{fontenc}
\usepackage{lmodern}
\usepackage[cmtip,all]{xy}
\usepackage{doi}

\newtheorem{thm}{Theorem}[section]
\newtheorem{prop}[thm]{Proposition}
\newtheorem{lem}[thm]{Lemma}
\newtheorem{cor}[thm]{Corollary}
\newtheorem{IntroThm}{Theorem}
\newtheorem{IntroCor}[IntroThm]{Corollary}
\theoremstyle{definition}
\newtheorem{Def}[thm]{Definition}
\theoremstyle{remark}
\newtheorem{rem}[thm]{Remark}
\newtheorem{ex}[thm]{Example}

\numberwithin{equation}{section}

\usepackage{mathrsfs}

\newcommand{\sI}{{\mathscr I}}

\newcommand{\sL}{{\mathscr L}}

\newcommand{\sO}{{\mathscr O}}

\newcommand{\sX}{{\mathscr X}}
\newcommand{\sY}{{\mathscr Y}}

\newcommand{\C}{{\mathbf C}}

\newcommand{\G}{{\mathbf G}}

\renewcommand{\P}{{\mathbf P}}
\newcommand{\Q}{{\mathbf Q}}
\newcommand{\R}{{\mathbf R}}

\newcommand{\Z}{{\mathbf Z}}
\renewcommand{\epsilon}{\varepsilon}
\renewcommand{\phi}{\varphi}
\newcommand{\CH}{{\rm CH}}

\newcommand{\Pic}{{\rm Pic}}

\newcommand{\Spec}{\operatorname{Spec}}

\newcommand{\GL}{{\operatorname{\rm GL}}}
\renewcommand{\deg}{{\mathrm{deg}}}
\renewcommand{\L}{\mathbf{L}}
\newcommand{\ind}{\mathrm{ind}}
\renewcommand{\gcd}{\mathrm{gcd}}

\newcommand{\td}{\mathrm{td}}
\newcommand{\mult}{\mathrm{mult}}
\newcommand{\ch}{\mathrm{ch}}
\newcommand{\Rep}{\mathrm{Rep}}
\newcommand{\sig}{\mathrm{sig}}
\newcommand{\Laz}{\mathbf{Laz}}
{\setbox0\hbox{$ $}}\fontdimen16\textfont2=\fontdimen17\textfont2
\makeatletter
\def\myrightarrow{{\setbox\z@\hbox{$\rightarrow$}\dimen0\ht\z@\multiply\dimen0 8\divide\dimen0 10\ht\z@\dimen0\box\z@}}
\def\myrightarrowfill@{\arrowfill@\relbar\relbar\myrightarrow}
\newcommand{\myxrightarrow}[2][]{\ext@arrow 0359\myrightarrowfill@{#1}{#2}}
\makeatother
\newcommand{\isoto}{\myxrightarrow{\mkern5mu\sim\mkern5mu}}

\begin{document}

\title[Index of varieties over Henselian fields]{Index of varieties over Henselian fields and Euler characteristic of coherent sheaves}
\author{H\'el\`ene Esnault}
\author{Marc Levine}
\author{Olivier Wittenberg}
\address{
FU Berlin\\
Mathematik\\
Arnimallee 3\\
14195 Berlin\\
Germany}
\email{esnault@math.fu-berlin.de}
\address{Universit\"at Duisburg-Essen\\
Fakult\"at Mathematik, Campus Essen\\
45127 Essen\\ 
Germany}
\email{marc.levine@uni-due.de}
\address{D\'epartement de math\'ematiques et applications\\
\'Ecole normale sup\'erieure\\
45 rue d'Ulm\\
75230 Paris Cedex 05\\
France}
\email{wittenberg@dma.ens.fr}
\thanks{The first author is supported by the Einstein Foundation, the ERC Advanced
Grant 226257 and the Chaire d'Excellence 2011 de la Fondation Sciences
Math\'ematiques de Paris, the second author by the Alexander von Humboldt
Foundation}


\renewcommand{\abstractname}{Abstract}
\begin{abstract}
Let~$X$ be a smooth proper variety over the quotient field of
a Henselian discrete valuation ring with algebraically closed residue field
of characteristic~$p$.
We show that for any coherent sheaf~$E$ on~$X$, the index of~$X$ divides the Euler--Poincar\'e characteristic $\chi(X,E)$ if $p=0$ or $p>\dim(X)+1$.
If $0<p\leq \dim(X)+1$, the prime-to\nobreakdash-$p$ part of the index of~$X$ divides $\chi(X,E)$.
Combining this with the Hattori--Stong theorem yields an analogous result concerning the divisibility of the cobordism class of~$X$ by the index of~$X$.

As a corollary, rationally connected varieties over the maximal unramified extension of a $p$\nobreakdash-adic field possess a zero-cycle
of $p$\nobreakdash-power degree (a zero-cycle of degree~$1$ if $p>\dim(X)+1$).
When $p=0$, such statements also have implications for the possible
multiplicities of singular fibers in degenerations of complex projective varieties.
\end{abstract}

\date{September 26, 2012; revised on January 9, 2013}
\maketitle

\section*{Introduction}

The \emph{index} of a variety~$X$ over a field~$K$ is the
smallest positive degree of a zero-cycle on~$X$, or equivalently, the
greatest common divisor
of the degrees $[L:K]$
of all finite extensions $L/K$ such that $X(L)\neq\emptyset$.
The present paper is devoted to investigating this invariant when~$K$ is the quotient field of an excellent Henselian discrete valuation ring with
algebraically closed residue field, for instance
$\Q_p^{\mathrm{nr}}$, the maximal unramified
extension of~$\Q_p$, or
the field of formal Laurent
series $\C((t))$.

Such fields are~$(C_1)$ fields in the sense of
Lang~\cite{Lang}.
Recall that a field~$K$ is said to be~$(C_1)$ if every hypersurface of degree $d\leq n$ in~$\P^n_K$ possesses a rational point.
Smooth hypersurfaces of degree $d \leq n$ in $\P^n_K$ are examples of smooth Fano varieties,
in particular they are rationally chain connected (and thus separably rationally connected if~$K$ has characteristic~$0$).
It is an old question raised by Lang, Manin, and Koll\'ar, whether
smooth proper varieties over~$K$
which are either Fano or separably rationally connected
always have a rational point
if~$K$ is $(C_1)$
(see~\cite{L}, \cite[p.~48, Remark~2.6 (ii)]{Ma}).  A positive answer is known to hold
when~$K$ is a finite field (see \cite{Es}).
In the separably rationally connected case, a positive answer also holds
when~$K$ is the function field of a curve over an algebraically closed field
(see \cite{GHS}, \cite{dJS}), from which it follows, by a global-to-local approximation argument, that a positive answer holds when~$K$ is the quotient field of
an equal characteristic Henselian discrete valuation ring with algebraically closed residue field
(see \cite[Th\'eor\`eme~7.5]{CT}).  In the local situation, no direct proof is known.  On the other hand, over function fields of curves,
the arguments of~\cite{GHS} and~\cite{dJS} rely on the study of moduli spaces of rational curves and are thus anchored in geometry; in particular,
they shed no light on unequal characteristic local fields with algebraically closed residue fields, such as~$\Q_p^{\mathrm{nr}}$.  It~is still unknown whether any smooth proper rationally connected variety over~$\Q_p^{\mathrm{nr}}$ possesses a rational point.

We can ask for less by considering the index.
As the index of a proper variety~$X$ over~$K$ is an invariant of cohomological nature (being determined by the cokernel of the degree map $\CH_0(X) \to \Z$ from
the Chow group of zero-cycles),
one would expect that cohomological conditions on~$X$, rather than the actual geometry of~$X$, should suffice to control the index. 

Indeed, various authors have shown that smooth proper varieties over $K=\C((t))$ which satisfy $H^i(X,\sO)=0$ for all $i>0$ have index~$1$
(see \cite[p.~162]{Mori}, \cite[p.~194]{Ko2}, \cite{NicaiseLetter}, \cite[Proposition~7.3]{CTV}; the last three references rely on Hodge theory).
This statement applies in particular to rationally connected varieties.
On the other hand, it is a well-known consequence of the adjunction formula
for surfaces that a family of curves $f:X\to S$,
where~$X$ is a smooth surface and~$S$ is a smooth curve,
has no multiple fiber if the generic fiber of~$f$ is a smooth, geometrically irreducible curve of genus~$2$.
Equivalently, every smooth proper curve of genus~$2$ over $\C((t))$ has index~$1$,
even though the groups $H^i(X,\sO)$ do not vanish in this case.

In this paper, we show that these two statements are in fact instances of a single general phenomenon relating the index of~$X$ over~$K$ and the Euler--Poincar\'e characteristic
of coherent sheaves on~$X$.  Note that $|\chi(X,\sO_X)|=1$ when~$X$ is either a rationally connected variety or a curve of genus~$2$.

\begin{IntroThm}[see Theorem~\ref{th:dividesregular}, Theorem~\ref{th:B}]
\label{thm01}
Let~$R$ be a Henselian discrete valuation ring with quotient field~$K$ and algebraically closed residue field~$k$.
Let~$X$ be a proper scheme over~$K$.
Assume~$k$ has characteristic~$0$.
Then, for any coherent sheaf~$E$ on~$X$, the index of~$X$ over~$K$ divides $\chi(X,E)$.
\end{IntroThm}

The assumption that~$k$ have characteristic~$0$ is required to
ensure that resolution of singularities holds for integral schemes of
finite type over~$R$.  For the arguments of Theorem~\ref{thm01} to go through in general,
it would suffice to know that
for any integral proper $K$\nobreakdash-scheme~$X$,
there is a normal proper flat $R$\nobreakdash-scheme~$\sX$ and a birational $K$\nobreakdash-morphism
$\sX \otimes_R K \to X$
such that the special fiber $\sX \otimes_R k$ is
divisible, as a Cartier divisor on~$\sX$, by its multiplicity as a Weil
divisor.
We cannot prove the existence of such models
when~$k$ has positive characteristic.
Nonetheless, using
Gabber's refinement of de Jong's
theorem on alterations, together with
a $K$\nobreakdash-theoretic d\'evissage and
the Hirzebruch--Riemann--Roch
theorem, we show:

\begin{IntroThm}[see Theorem~\ref{thm:Divisibility}]
\label{thm02}
Let~$R$ be a Henselian discrete valuation ring with quotient field~$K$ and algebraically closed residue field~$k$.
Let~$X$ be a smooth proper scheme over~$K$.
Assume~$k$ has characteristic $p>0$.
Then, for any coherent sheaf~$E$ on~$X$, the prime-to\nobreakdash-$p$ part of the index of~$X$ over~$K$ divides $\chi(X,E)$.
If in addition~$p>\dim(X)+1$, the index of~$X$ over~$K$ divides $\chi(X,E)$.
\end{IntroThm}

(By the \emph{prime-to\nobreakdash-$p$ part of~$N$}, we mean the largest divisor
of~$N$ which is prime to~$p$.)
When~$R$ is excellent, the smoothness assumption in Theorem~\ref{thm02} may be
removed (see Theorem~\ref{th:B} and Remark~\ref{rk:haution}).
Coming back to the motivation for our work,
we deduce from Theorem~\ref{thm02}:

\begin{IntroCor}[see Corollary~\ref{cor:RC}]
\label{RC}
The index of a rationally connected variety~$X$ over~$\Q_p^{\mathrm{nr}}$,
or more generally over
the quotient field of a Henselian discrete valuation ring of characteristic~$0$ with algebraically closed residue field of characteristic $p>0$,
is a power of~$p$.
It is~$1$
if in addition $p>\dim(X)+1$.
\end{IntroCor}

Theorems~\ref{thm01} and~\ref{thm02} also have a number of unexpected consequences.
For instance, we prove that if~$X$ is a variety of general type
over $\C((t))$, or over the maximal unramified extension of a
$p$\nobreakdash-adic field with $p>\dim(X)+1$, then the index of~$X$
divides the plurigenera $P_n(X)$ for $n \geq 2$
(Example~\ref{ex:generaltypesurface}, Corollary~\ref{cor:generaltype}).
In another direction,
if $K=\C((t))$ or~$K$ is the maximal unramified
extension of a $p$\nobreakdash-adic field with $p\geq 5$,
then hypersurfaces of degree~$6$ in $\P^3_K$ have index~$1$.
More generally, in Section~\ref{sec:Hyp}, we produce, for any~$d$ and~$N$,
an optimal bound on the index of a hypersurface of degree~$d$ in $\P^N$ over such a field
(Theorem~\ref{thm:hypersurf}, Proposition~\ref{prop:twistedfermat}).
In particular,
if $d=\prod p_i^{\alpha_i}$ is the prime factorization of~$d$,
the property ``every hypersurface of degree~$d$ in $\P^N$ over $\C((t))$ contains a zero-cycle of degree~$1$''
holds if and only if $\max(p_1^{\alpha_1},\dots,p_n^{\alpha_n}) \leq N$
(Example~\ref{ex:hypcont}).  This should be compared with Lang's theorem according to which $\C((t))$ is a~$(C_1)$ field.

Let us stress the geometric content of such statements: they imply in
particular that if $f:X \to S$ is a dominant morphism between smooth
projective complex varieties and if the geometric generic fiber of~$f$ is
an irreducible variety of general type (resp., is a sextic surface), then
the multiplicities of the codimension~$1$ fibers of~$f$ divide the higher plurigenera of
the generic fiber (resp., the morphism $f$ has no multiple fiber).

\bigskip
Although we use a slightly different method here, Theorem~\ref{thm01} may be proved using
only properties of $K$\nobreakdash-theory as an oriented cohomology theory (in the sense of \cite[Definition 1.1.2]{AlgCobord}), which,
assuming~$K$ is a subfield of~$\C$, suggests that the class of $X(\C)$ in the complex cobordism ring
$\pi_*(MU)$ should be divisible by the index of~$X$ over~$K$.
The goal of Section~\ref{sec:Cobordism} is to show that this is indeed the case, and to prove a similar statement without any assumption on the characteristic of~$K$.

As it turns out, when~$K$ is a subfield of~$\C$,
one can give a purely algebraic description of the complex cobordism class of $X(\C)$.
It has long been known (see, \emph{e.g.}, \cite[Chapter I]{Stong})
that
 the complex cobordism ring $\pi_*(MU)$
is canonically isomorphic to
the graded subring~$\L$ of the infinite polynomial ring
$\Z[b_1,b_2,\ldots]$ (with $\deg(b_i)=i$) spanned by the polynomials
$b(X)=\sum_{|I|=\dim(X)} \deg(c_I(-T_X))b^I$ as~$X$ runs over smooth
proper varieties over~$\C$, and
that~$b(X)$ is equal to the cobordism class of~$X(\C)$ via this
isomorphism.  Here~$T_X$ denotes the tangent bundle of~$X$ and~$c_I$
is the $I$th Conner--Floyd Chern class.
When~$K$ is an arbitrary field, we take this description as our point of departure and define the cobordism ring of $\Spec(K)$
to be the subring~$\L_K$ of $\Z[b_1,b_2,\ldots]$ spanned by the polynomials
$b_K(X)=\sum_{|I|=\dim(X)} \deg(c_I(-T_X))b^I$ when~$X$ runs over smooth
proper varieties over~$K$. This subring is in fact equal to~$\L$, according to
Merkurjev~\cite[Theorem~8.2]{MerkurjevOrient}.
Under the hypotheses of Theorem~\ref{thm02}, we may ask whether the cobordism class $b_K(X)$
is divisible, in~$\L_K$, by the index of~$X$ over~$K$.
In this direction, we show:

\begin{IntroThm}[see Theorem~\ref{thm:Cobordism}]
\label{thm04}
Let~$R$ be a Henselian discrete valuation ring with
quotient field~$K$, and algebraically closed residue field~$k$ of characteristic $p\geq 0$.
Let~$X$ be a smooth proper irreducible $K$\nobreakdash-scheme.
If $p=0$ or $p>\dim(X)+1$, then $b_K(X)$ is divisible, in the ring~$\L_K$, by the index of~$X$ over~$K$.
If $p>0$, then $b_K(X)$ is divisible, in~$\L_K$, by the prime-to\nobreakdash-$p$ part of the index of~$X$ over~$K$.
\end{IntroThm}

Since a full theory of algebraic cobordism is only available in
characteristic zero, we use another method to prove Theorem~\ref{thm04},
namely the Hattori--Stong theorem \cite[Theorem 1]{StongRR}, \cite[Theorem~I]{Hattori}, which allows one to use $K$\nobreakdash-theory to compute
cobordism; thus Theorem~\ref{thm04} becomes a consequence of Theorem~\ref{thm02}.

From Theorem~\ref{thm04} and from the well-known fact that $\pi_*(MU)$ is generated
by the classes of projective spaces $\P^n$ ($n \geq 1$) and Milnor hypersurfaces $H_{m,n}$ (hypersurfaces
of bidegree $(1,1)$ in $\P^m \times \P^n$, with $2\leq m\leq n$), we deduce the following concrete consequence regarding integral-valued rational characteristic classes:

\begin{IntroCor}[see Corollary~\ref{cor:Cobordismcn}]
\label{cor05}
Let~$R$ be a Henselian discrete valuation ring with
quotient field~$K$, and algebraically closed residue field~$k$ of characteristic $p\geq 0$.
Let $d \geq 1$ and let $P \in \Q[c_1,\dots,c_d]$ be homogeneous of degree~$d$ with respect to the grading $\deg(c_i)=i$.
Assume that $\deg(P(c_1(T_X),\dots,c_d(T_X)))\in \Z$ for any $d$\nobreakdash-dimensional product~$X$ of complex projective spaces and Milnor hypersurfaces.
Then, for any smooth proper irreducible $K$\nobreakdash-scheme $X$ of dimension~$d$, the rational number $\deg(P(c_1(T_X),\dots,c_d(T_X)))$ is an integer.
If $p=0$ or $p>\dim(X)+1$, this integer is divisible by the index of~$X$ over~$K$.
If $p>0$, it is divisible by the prime-to\nobreakdash-$p$ part of the index of~$X$ over~$K$.
\end{IntroCor}

At the end of Section~\ref{sec:Cobordism} we list examples of such
polynomials, such as $\frac{1}{2} c_d$ or $\frac{1}{2} c_1^d$ if~$d$ is
odd.  As pointed out by Merkurjev in \cite[p.~8]{MerkurjevSteenrodOps},
properties of Brosnan's Steenrod operations on the mod $q$ Chow groups show
that, for a given prime number~$q$, the characteristic class $X\mapsto \frac{1}{q} \deg(c_I(-T_X))$
is integral-valued as long as $I=(\alpha_j)_{j \geq 1}$ with $\alpha_j=0$ whenever~$j$ is not of the form $q^n-1$ for some natural number~$n$.
This gives a large supply of integral-valued rational characteristic classes that are not expressible as
$\Z$\nobreakdash-linear combinations of monomials in the Chern classes of the tangent bundle.

Although the Hattori--Stong theorem tells us that each integral-valued
characteristic class $X\mapsto \deg(P(c_1(T_X),\dots,c_d(T_X)))$ on $d$\nobreakdash-dimensional smooth
proper varieties is given as the Euler--Poincar\'e characteristic of $\rho(T_X)$
for some virtual representation~$\rho$ of $\GL_d$, we are not aware of any
general and explicit formulas for~$\rho$ in terms of~$P$.  While one can
find a~$\rho$ for the characteristic class $\frac{1}{2} c_d$ (for odd~$d$), one
cannot do this so easily for other classes, for example for $\frac{1}{2}c_1^d$. The same seems to be the case for the series of characteristic
classes $X\mapsto \frac{1}{q} \deg(c_I(-T_X))$ mentioned above.  Thus
Theorem~\ref{thm04} yields nontrivial divisibility by the index for
integral-valued rational characteristic classes which, at least as a
practical matter, goes beyond Theorem~\ref{thm02}.

\bigskip
{\it Acknowledgements.}
We thank Markus Rost and Alexander Merkurjev for directing us to the Hattori--Stong theorem and for their comments and suggestions,
\mbox{Jean-Louis} Colliot-Th\'el\`ene for pointing out the reference~\cite{Atiyah}, Johannes Nicaise for his interest and for discussions on the topic of this note,
and the referee for bringing the paper~\cite{Haution} to our attention.

\bigskip
{\it Notation.} 
If~$N$ and~$n$ are integers, the \emph{prime-to\nobreakdash-$N$ part of~$n$} is the largest integer which divides~$n$ and is prime to~$N$
(or~$0$ if $n=0$).
Let~$X$ be a scheme of finite type over a field~$K$.
If~$X$ is smooth over~$K$, we denote by~$T_X$ the tangent bundle of~$X$, \emph{i.e.}, the locally free sheaf dual to $\Omega^1_{X/K}$.
If~$X$ is proper over~$K$ and~$E$ is a coherent sheaf on~$X$, we denote the Euler--Poincar\'e characteristic of~$E$ by
$\chi(X, E)=\sum_{i\ge0}(-1)^i\dim_K H^i(X, E)$.
Finally, if~$R$ is a discrete valuation ring with quotient field~$K$,
a \emph{model} of~$X$ over~$R$ is a flat $R$\nobreakdash-scheme of finite type with generic fiber~$X$.

\smallskip\section{Index of smooth proper schemes over arbitrary fields}
\label{sec:all_fields}

\begin{Def}
Let~$X$ be a scheme of finite type over a field~$K$.  The \emph{index of~$X$ over~$K$},
denoted $\ind(X)$, is the greatest common divisor of the degrees $[K(x):K]$ of the closed points~$x$ of~$X$.
\end{Def}

The index of~$X$ over~$K$ is also characterized by the fact that it generates the subgroup $\deg(Z_0(X)) \subseteq \Z$.
By the covariant functoriality of~$Z_0(X)$,
it follows that $\ind(X)$ divides $\ind(Y)$ for any morphism $Y \to X$ of schemes of finite type over~$K$.

\begin{prop}
\label{prop:RR}
Let~$X$ be a smooth proper $K$\nobreakdash-scheme of dimension~$d$
and~$E$ be a coherent sheaf on~$X$.
Then the prime-to\nobreakdash-$(d+1)!$ part of $\ind(X)$ divides $\chi(X,E)$.
\end{prop}

\begin{proof}
As~$X$ is regular and separated, the Grothendieck group of coherent sheaves on~$X$ is generated by the classes of locally free sheaves.
Thus, we may assume that~$E$ is locally free.
According to the Hirzebruch--Riemann--Roch theorem,
we then have an equality of rational numbers
\begin{align}
\label{eq:RR}
\chi(X,E) = \epsilon_*\big(\ch(E) \cdot \td(T_X)\big) \in \CH^*(\Spec(K)) \otimes_\Z \Q =\Q\rlap{\text{,}}
\end{align}
where $\epsilon:X \to \Spec(K)$ is the structure morphism of~$X$,
and $\ch(E)$, $\td(T_X) \in \CH^*(X) \otimes_\Z \Q$ respectively denote the Chern character of~$E$ and the Todd class of the tangent bundle~$T_X$ of~$X$
(see \cite[Corollary~15.2.1]{Fulton}).

\begin{lem}
\label{lem:RRdenom}
Let~$X$ be a smooth proper irreducible $K$\nobreakdash-scheme of dimension~$d$ and 
let~$E$ be a vector bundle on~$X$.
The Chern character $\ch(E) \in \CH^*(X) \otimes_\Z \Q$ belongs to the image of $\CH^*(X) \otimes_\Z \Z[1/d!]$.
The Todd class $\td(E) \in \CH^*(X) \otimes_\Z \Q$ belongs to the image of $\CH^*(X) \otimes_\Z \Z[1/(d+1)!]$.
\end{lem}

\begin{proof}
Let~$r$ denote the rank of~$E$.
Let $\ch_{r,d} \in \Q[\xi_1,\dots,\xi_r]$ 
(resp., $\td_{r,d} \in \Q[\xi_1,\dots,\xi_r]$)
denote the degree~$d$ polynomial obtained by truncating the formal power series $\sum_{j=1}^r \exp(\xi_j)$
(resp., 
 $\prod_{j=1}^r \frac{\xi_j}{1-\exp(-\xi_j)}$).
This polynomial has coefficients in $\Z[1/d!]$ (resp., in $\Z[1/(d+1)!]$),
has degree~$d$,
and is invariant under all permutations of the~$\xi_j$'s.  Therefore
it belongs to the subring $\Z[1/d!][c_1,\dots,c_d]$
(resp., $\Z[1/(d+1)!][c_1,\dots,c_d]$) of $\Q[\xi_1,\dots,\xi_r]$,
where $c_1,\dots,c_r$ denote the elementary symmetric polynomials in the $\xi_j$'s
(with the convention that $c_i=0$ if $i>r$).
Now the Chern character of~$E$ (resp., the Todd class of~$E$) is, by definition, the element of $\CH^*(X) \otimes_\Z \Q$ obtained
by applying $\ch_{r,d}$ (resp., $\td_{r,d}$) to the Chern classes $c_i(E)\in \CH^*(X)$ of~$E$; hence the lemma.
\end{proof}

Thanks to Lemma~\ref{lem:RRdenom}, it follows from~\eqref{eq:RR} that $\chi(X,E)$ belongs to the image of
$\epsilon_*:\CH^*(X) \otimes_{\Z} \Z[{1}/{(d+1)!}] \to \CH^*(\Spec(K)) \otimes_\Z \Z[{1}/{(d+1)!}]=\Z[{1}/{(d+1)!}]$.
In other words, there exist an integer $m \geq 0$ and a zero-cycle~$z$ on~$X$ such that
$$((d+1)!)^m \chi(X,E)=\deg(z)\rlap{\text{,}}$$
which proves the proposition.
\end{proof}

\begin{rem}
\label{rk:haution}
Haution~\cite[Theorem~5.1~(ii)]{Haution} has proved that Proposition~\ref{prop:RR} remains valid without any smoothness assumption on~$X$.
\end{rem}

\smallskip\section{Index of the generic fiber of a regular proper scheme over a Henselian discrete valuation ring}

The goal of this section is to prove the following theorem:

\begin{thm}
\label{th:dividesregular}
Let~$R$ be a Henselian discrete valuation ring with quotient field~$K$ and algebraically closed residue field~$k$.
Let~$X$ be a regular proper $K$\nobreakdash-scheme.
Assume~$X$ admits a regular proper model over~$R$.
Then $\ind(X)$ divides $\chi(X,\sO_X)$.
\end{thm}

We first recall two elementary facts relating the index of the generic fiber
and the multiplicity of the special fiber of a scheme over a discrete valuation ring.

\begin{Def}
The \emph{multiplicity} of a noetherian scheme~$S$, denoted $\mult(S)$, is the greatest common divisor,
over all points $\eta \in S$ of codimension~$0$, of the length of the artinian local ring $\sO_{S,\eta}$.
\end{Def}

\begin{lem}
\label{lem:indexmult}
Let~$R$ be a discrete valuation ring with quotient field~$K$ and residue field~$k$.
Let $\sX$ be a flat $R$\nobreakdash-scheme of finite type, with generic fiber $X=\sX \otimes_R K$ and special fiber $Y=\sX \otimes_R k$.
\begin{enumerate}
\item[(i)] If~$R$ is Henselian and~$k$ is algebraically closed, then $\ind(X)$ divides $\mult(Y)$.
\item[(ii)] If~$\sX$ is locally factorial (e.g., regular) and
is proper over~$R$, then $\mult(Y)$ divides $\ind(X)$.
\end{enumerate}
\end{lem}

\begin{proof}
The first assertion follows from~\cite[\textsection9.1, Corollary~9]{BLR}.
The second assertion was included for the sake of completeness; we shall not use it.
It can be proved with a simple computation in intersection theory,
see \cite[\textsection8.1]{GabberLiuLorenzini} for details.
\end{proof}

\begin{proof}[Proof of Theorem~\ref{th:dividesregular}]
Let $\sX$ be a regular proper model of~$X$ over~$R$.
It is clear, from the definition of $\mult(Y)$,
that~$Y$ is divisible by $\mult(Y)$ as a Weil divisor on~$\sX$.
As~$\sX$ is regular, it follows that~$Y$ is divisible by $\mult(Y)$ as a Cartier divisor on~$\sX$.
Thus, Theorem~\ref{th:dividesregular} results from the combination of
Lemma~\ref{lem:indexmult}~(i)
and Proposition~\ref{prop:DivCartier} below.
\end{proof}

\begin{prop}
\label{prop:DivCartier}
Let~$R$ be a discrete valuation ring with quotient field~$K$.
Let~$\sX$ be a normal proper flat $R$\nobreakdash-scheme, with generic fiber~$X$ and special fiber~$Y$.
Let $m \geq 1$ be an integer such that~$Y$ is divisible by~$m$ as a Cartier divisor on~$\sX$.
Then~$m$ divides $\chi(X,\sO_X)$.
\end{prop}

\begin{proof}
By assumption, there is a Cartier divisor~$D$ on~$\sX$ satisfying the
equality of Cartier divisors $Y=mD$.  As~$Y$ is effective and~$\sX$ is
normal, $D$ is effective, so that $\sO_\sX(-D)$ is a sheaf of ideals
of~$\sO_\sX$.
For $j \in \{0,\dots,m\}$, the closed subscheme~$Y_j$ of~$\sX$ defined by the ideal sheaf $\sO_\sX(-jD)$
is contained in~$Y_m=Y$ and may thus be regarded as a closed subscheme of~$Y$.
The corresponding ideal sheaf $\sI_j \subseteq \sO_Y$
fits into an exact sequence
of $\sO_\sX$\nobreakdash-modules
\begin{align*}
\xymatrix{
0 \ar[r] & \sO_\sX(-mD) \ar[r] & \sO_\sX(-jD) \ar[r] & i_{Y*}\sI_j \ar[r] & 0\rlap{\text{,}}
}
\end{align*}
where~$i_Y$ denotes the inclusion of~$Y$ in~$\sX$.

Let $i: D \hookrightarrow Y$ and $i_D:D \hookrightarrow \sX$ denote the canonical closed immersions,
where in an abuse of notation~$D$ stands for the scheme~$Y_1$.
For every $j \in \{0,\dots,m-1\}$, there is an exact sequence of $\sO_Y$\nobreakdash-modules
\begin{align*}
\xymatrix{
0 \ar[r] & \sI_{j+1} \ar[r] & \sI_j \ar[r] & i_*(\sL^{\otimes j}) \ar[r] & 0\rlap{\text{,}}
}
\end{align*}
with $\sL=i_D^*\big(\sO_{\sX}(-D)\big)$.
As $\sI_0=\sO_Y$ and $\sI_m=0$, we deduce that
\begin{align}
\label{eq:chiY}
\chi(Y,\sO_Y)=\sum_{j=0}^{m-1} \chi(D,\sL^{\otimes j})\rlap{\text{.}}
\end{align}
According to Kleiman's version of Snapper's theorem
\cite[Chapter~I, \textsection1]{Kleiman}, there exists a polynomial $P \in \Q[X]$
such that $P(n)=\chi(D,\sL^{\otimes n})$ for all $n \in \Z$.
The $\sO_\sX$\nobreakdash-module $\sO_{\sX}(-mD)$ is free
since~$mD$ is a principal divisor on~$\sX$. Hence $\sL^{\otimes m}$ is a free invertible sheaf on~$D$.
In particular, the polynomial~$P$ takes the value $\chi(D,\sO_D)$ on all integer multiples of~$m$.
It is therefore a constant polynomial, equal to $\chi(D,\sO_D)$;
we conclude, thanks to~\eqref{eq:chiY}, that $\chi(Y,\sO_Y)=m\chi(D,\sO_D)$.
As $\chi(X,\sO_X)=\chi(Y,\sO_Y)$ (see \cite[Chapter~II, \textsection5, Corollary]{Mumford}), the proposition is proved.
\end{proof}

\begin{cor}\label{cor:Char0}
Let~$R$ be a Henselian discrete valuation ring with algebraically closed residue field of characteristic~$0$.  Let~$X$ be a smooth proper variety over the quotient field~$K$ of~$R$.
Then $\ind(X)$ divides $\chi(X,\sO_X)$.
\end{cor}

\begin{proof}
Resolution of singularities for integral $R$\nobreakdash-schemes of finite type holds, by~\cite[Theorem~1.1]{Temkin},
so that Theorem~\ref{th:dividesregular} may be applied.
\end{proof}

\begin{ex}
\label{ex:hyperelliptic}
A smooth proper curve of genus~$2$ over $\C((t))$ has index~$1$.  More generally, a hyperelliptic curve of even genus over $\C((t))$ has index~$1$.
These two assertions were already known (see Remark~\ref{rem:atiyah} below, or \cite[Theorem~2]{GPP}).
By Theorem~\ref{th:dividesregular},
they also hold for curves defined over the maximal unramified extension of a $p$\nobreakdash-adic field, as regular proper models exist in this case (see \cite{Abhyankar}).
\end{ex}

\begin{ex}
\label{ex:generaltypesurface}
If~$X$ is a smooth proper surface of general type over $\C((t))$, then $\ind(X)$ divides the plurigenera $P_n=\dim H^0(X,\sO_X(nK_X))$ for all $n \geq 2$.

Indeed, if the canonical class~$K_X$ is nef, then
\begin{align*}
P_n=\frac{n(n-1)}{2} (K_X)^2 + \chi(X,\sO_X)
\end{align*}
for $n \geq 2$ (see \cite[\textsection4]{Kodaira}, \cite[Ch.~10, Proposition~10]{FriedmanAlgSurfaces}).
In general, one can always find a smooth proper surface~$X'$ over $\C((t))$, birationally
equivalent to~$X$, with nef canonical class (see \cite[Ch.~III, Theorem~2.2]{KollarBook}).  This implies the claim as both
the plurigenera and the index are birational invariants among smooth proper varieties (see 
\cite[\textsection7.1]{Debarre}, \cite[Proposition~6.8]{GabberLiuLorenzini}).

Let us note that $\ind(X)$ does \emph{not} necessarily divide the first plurigenus $P_1=p_g$.
Indeed, smooth quintic surfaces in~$\P^3$ satisfy $p_g=4$, and
their index over $\C((t))$ may be equal to~$5$, as will result from Proposition~\ref{prop:twistedfermat} below.
\end{ex}

\begin{rem}
\label{rem:atiyah}
Let~$C$ be a smooth proper curve over $\C((t))$.
By Corollary~\ref{cor:Char0}, the index of~$C$ divides $\deg(K_C)/2$,
where~$K_C$ denotes a canonical divisor.
A~stronger statement is known to hold:
even the class of~$K_C$ in $\Pic(C)$ is divisible by~$2$
(see Atiyah~\cite[p.~61]{Atiyah}, which rests on results of Serre and Mumford).
One might wonder whether the same phenomenon occurs in higher dimension.
Namely, in the situation of Corollary~\ref{cor:Char0}, does the
$0$\nobreakdash-dimensional component of the Todd class $\td(T_X) \in \CH^*(X) \otimes_\Z \Q$ belong to the image of $\CH_0(X) \to \CH_0(X) \otimes_\Z \Q$?
\end{rem}

\smallskip\section{Index of proper schemes over the quotient field of an excellent Henselian discrete valuation ring}

The following theorem, which builds upon Theorem~\ref{th:dividesregular},
extends Corollary~\ref{cor:Char0} in three directions: the residue field of~$R$ may have positive characteristic,
the scheme~$X$ may be singular,
and any coherent sheaf may be used instead of~$\sO_X$.

Recall that a discrete valuation ring~$R$ is \emph{excellent} if
the field extension $\hat K/K$ is separable, where~$K$ denotes the quotient
field of~$R$ and~$\hat K$ is its completion.  This condition is trivially satisfied
if~$K$ has characteristic~$0$ or~$R$ is complete.

\begin{thm}
\label{th:B}
Let~$R$ be an excellent Henselian discrete valuation ring with algebraically closed residue field~$k$ and quotient field~$K$.  Let~$X$ be a proper $K$\nobreakdash-scheme.
Let~$E$ be a coherent sheaf on~$X$.
\begin{enumerate}
\item[(i)] If~$k$ has characteristic~$0$, then $\ind(X)$ divides $\chi(X,E)$.
\item[(ii)] If~$k$ has characteristic $p>0$, the prime-to\nobreakdash-$p$ part of $\ind(X)$ divides $\chi(X,E)$.
\end{enumerate}
\end{thm}

\begin{proof}
The Grothendieck group of coherent sheaves on~$X$ is generated by the
classes of~$\sO_Z$ for all integral closed subschemes~$Z$ of~$X$ (see
\cite[\textsection8, Lemme~17]{BorelSerre} and
\cite[Proposition~1.1]{BerthelotSGA6}).  Moreover, the index of~$X$ divides the index of any subscheme of~$X$.
Therefore it suffices to
show that $\ind(Z)$, or the prime-to\nobreakdash-$p$ part of $\ind(Z)$, divides
$\chi(Z,\sO_Z)$, for all integral closed subschemes~$Z$ of~$X$.

This remark proves the theorem in case $\dim(X)=0$.
Indeed, for any integral closed subscheme~$Z$ of dimension~$0$, we have
$\ind(Z)=\deg(Z)=\chi(Z,\sO_Z)$.
To establish the theorem in general, we argue by induction on $d=\dim(X)$.
Assume that the conclusion of the theorem holds for all proper $K$\nobreakdash-schemes of dimension~$<d$.
In order to prove that it also holds for~$X$, we may assume, thanks to
the preceding paragraph, that~$X$ is integral and that $E=\sO_X$.

Let~$\ell$ be a prime number invertible in~$k$.  We shall prove that the largest power of~$\ell$ which divides $\ind(X)$ also divides $\chi(X,\sO_X)$.
By Nagata's compactification theorem~\cite{Nagata}, there exists a proper model~$\sX$ of~$X$ over~$R$.
According to a theorem of Gabber and de Jong \cite[Theorem~1.4]{Illusie},
as~$R$ is excellent,
there exist a regular scheme~$\sY$, and a proper, surjective, generically finite morphism $h:\sY \to \sX$ of degree~$N$ prime to~$\ell$.
Let $Y = \sY \otimes_R K$ and let $f:Y \to X$ denote the morphism induced by~$h$.

Let $U \subseteq X$ be a dense open subset above which~$f$ is finite and flat.
Recall the exact sequence of Grothendieck groups of coherent sheaves
\begin{align}
\label{es:localization}
\xymatrix{
G_0(C) \ar[r] & G_0(X) \ar[r] & G_0(U) \ar[r] & 0\rlap{\text{,}}
}
\end{align}
where $C=X \setminus U$
(see \cite[\textsection8, Proposition~7]{BorelSerre}).
The restriction to~$U$ of the coherent sheaf $R^q f_* \sO_Y$ is~$0$ for $q>0$, and is a locally free $\sO_U$\nobreakdash-module of rank~$N$ for $q=0$,
which we may assume to be a free $\sO_U$\nobreakdash-module of rank~$N$ after further shrinking~$U$.
It then follows, thanks to~\eqref{es:localization}, that the class in~$G_0(X)$ of the bounded complex of coherent sheaves $Rf_*\sO_Y$ is the sum of the class of $\sO_X^N$
and a class coming from~$G_0(C)$.  Therefore there exists a virtual
coherent sheaf~$F$ on~$C$ such that
\begin{align}
\label{eq:chidevissage}
\chi(Y,\sO_Y)=N\chi(X,\sO_X)+\chi(C,F)\rlap{\text{.}}
\end{align}
As~$C$ is a subscheme of~$X$, the index of~$X$ divides the index of~$C$.  Therefore, thanks to the induction hypothesis, if $p=0$ then $\ind(X)$ divides $\chi(C,F)$, and if $p>0$, then the prime-to\nobreakdash-$p$ part
of~$\ind(X)$ divides $\chi(C,F)$.
Moreover, the index of~$X$ divides the index of~$Y$ because there exists a morphism from~$Y$ to~$X$.
The index of~$Y$ in turn divides $\chi(Y,\sO_Y)$ according to Theorem~\ref{th:dividesregular}.
All in all, we deduce
from~\eqref{eq:chidevissage} that the prime-to\nobreakdash-$Np$ part of $\ind(X)$ divides $\chi(X,\sO_X)$.
As~$\ell$ is prime to~$N$, this concludes the proof.
\end{proof}

Combining Theorem~\ref{th:B} and Proposition~\ref{prop:RR} yields:

\begin{thm}
\label{thm:Divisibility}
Let~$R$ be a Henselian discrete valuation ring with algebraically closed residue field~$k$ and quotient field~$K$.
Let~$X$ be a smooth proper $K$\nobreakdash-scheme.  Let~$E$ be a coherent sheaf on~$X$.
Let~$p$ denote the characteristic of~$k$.
\begin{enumerate}
\item[(i)] If $p=0$, then $\ind(X)$ divides $\chi(X,E)$.
\item[(ii)] If $p>0$, then the prime-to\nobreakdash-$p$ part of $\ind(X)$ divides $\chi(X,E)$.
\item[(iii)] If $p>\dim(X)+1$, then $\ind(X)$ divides $\chi(X,E)$.
\end{enumerate}
\end{thm}

\begin{proof}
In view of Lemma~\ref{lem:indexcompletion} below, we may assume that~$R$ is
complete and hence excellent.  In this case, we may apply
Theorem~\ref{th:B} and Proposition~\ref{prop:RR}.
\end{proof}

\begin{lem}
\label{lem:indexcompletion}
Let~$R$ be a Henselian discrete valuation ring, with completion $\hat R$.
Let~$K$ and $\hat K$ denote the quotient fields of~$R$ and $\hat R$ respectively.
If~$X$ is a smooth $K$\nobreakdash-scheme, the index of~$X$ over~$K$ and the index of $X \otimes_K \hat K$ over $\hat K$ are equal.
\end{lem}

\begin{proof}
Let~$P$ be a closed point of $X \otimes_K \hat K$.  Let~$d$ be its degree.  We must show that $\ind(X)$ divides~$d$.
For this we may assume that the residue field $\hat K(P)$ of~$P$ is separable over $\hat K$, by \cite[Theorem~9.2]{GabberLiuLorenzini}.
Under this assumption, there exists a finite extension $L/K$ such that $L \otimes_K \hat K =\hat K(P)$, by Krasner's lemma.
After replacing~$X$ with a neighborhood of~$P$, we may assume in addition
that~$X$ is quasi-projective over~$K$.
The Weil restriction of scalars $R_{L/K}(X \otimes_K L)$ is then a smooth $K$\nobreakdash-scheme
(see \cite[\textsection7.6, Theorem~4 and Proposition~5]{BLR}).
Applying \cite[\textsection3.6, Corollary~10]{BLR} to it now shows that $X(L) \neq \emptyset$.
\end{proof}

\begin{ex}
\label{ex:generaltypesurfacep}
According to Theorem~\ref{thm:Divisibility},
Example~\ref{ex:generaltypesurface} is also valid over the maximal unramified extension of a $p$\nobreakdash-adic field with $p \geq 5$.
\end{ex}

As immediate consequence of Theorem~\ref{thm:Divisibility}, we have:

\begin{cor}
\label{cor:chi1}
We keep the assumptions of Theorem~\ref{thm:Divisibility}.  Assume $|\chi(X,\sO_X)|=1$.
Then $\ind(X)$ is a power of~$p$. If moreover $p=0$ or $\dim(X)<p-1$,
then $\ind(X)=1$.
\end{cor}

Corollary~\ref{cor:chi1} applies to geometrically irreducible smooth varieties over~$K$ which satisfy $H^i(X,\sO_X)=0$ for $i>0$,
for instance Fano varieties or more generally rationally connected varieties when~$K$ has characteristic~$0$
(see~\cite[Ch.~IV]{KollarBook}; we say that a variety~$X$ over~$K$ is \emph{rationally connected} if $X \otimes_K {\bar K}$ is rationally
connected in the sense of \cite{KollarBook}, where $\bar K$ denotes an algebraic closure of~$K$).
We thus obtain:

\begin{cor}
\label{cor:RC}
Let~$X$ be a smooth proper rationally connected variety over the maximal
unramified extension of a $p$\nobreakdash-adic field.
The index of~$X$ is a power of~$p$.  If~moreover $\dim(X)<p-1$, the index of~$X$ is equal to~$1$.
\end{cor}

Corollary~\ref{cor:RC} gives some evidence for the conjecture
according to which any rationally connected variety over the maximal
unramified extension of a $p$\nobreakdash-adic field possesses a rational
point.

Theorem~\ref{thm:Divisibility} also has interesting consequences when applied to other coherent sheaves than~$\sO_X$.
As an illustration, we extend Examples~\ref{ex:generaltypesurface} and~\ref{ex:generaltypesurfacep} to
higher-dimensional general type varieties. (The argument of Example~\ref{ex:generaltypesurface} fails in higher
dimension because already a smooth proper threefold need not be birationally equivalent to any smooth proper threefold with nef canonical class.)

\begin{cor}
\label{cor:generaltype}
Let~$X$ be a smooth proper variety of general type over $\C((t))$
or over the maximal unramified extension of a $p$\nobreakdash-adic field with $p>\dim(X)+1$.
Then $\ind(X)$ divides the plurigenera $P_n=\dim H^0(X,\sO_X(nK_X))$ for all $n \geq 2$.
\end{cor}

\begin{proof}
According to Koll\'ar and Lazarsfeld,
for $n \geq 2$, the $n$th plurigenus of a smooth proper variety~$X$ of general type over a field of characteristic~$0$
may be expressed as the Euler--Poincar\'e characteristic of the coherent sheaf
$$\sO_X(nK_X) \otimes \sI(\|(n-1)K_X\mkern-1mu\|)\rlap{\text{,}}$$
where $\sI(\|(n-1)K_X\mkern-1mu\|)$ denotes the asymptotic multiplier ideal sheaf
associated to the complete linear system $|(n-1)K_X\mkern-1mu|$
(see \cite[\textsection11.2.C]{Lazarsfeld2}, \cite[Example~2.5]{KollarPairs}).
\end{proof}

\smallskip\section{Application: hypersurfaces in projective space}
\label{sec:Hyp}

We illustrate the results of the previous sections by examining the case of hypersurfaces in projective space.
It turns out that a simple application of Theorem~\ref{th:B} yields the best possible bound on the index of a degree~$d$ hypersurface
in $\P^N$ over the quotient field of an
excellent Henselian discrete valuation ring with algebraically closed
residue field, for any~$d$ and~$N$ (Theorem~\ref{thm:hypersurf} and
Proposition~\ref{prop:twistedfermat}).  For certain values of~$d$ and~$N$, this bound is equal to~$1$, thus yielding unexpected
existence results for zero-cycles of degree~$1$ (Examples~\ref{ex:hyp}~(i) and~\ref{ex:hypcont}).

Given two integers~$d$, $N$, let
\begin{align*}
I_{d,N}= \gcd \;\!\biggl\{ \frac{d}{\delta} \; ; \; \delta \in \{1,\dots,N\} \text{ and } \delta \text{ divides } d\biggr\}\rlap{\text{.}}
\end{align*}

\begin{thm}
\label{thm:hypersurf}
Let~$R$ be an excellent Henselian discrete valuation ring with algebraically closed residue field~$k$ and quotient field~$K$.
Let~$p$ denote the characteristic of~$k$.
For any hypersurface $X \subset \P^N_K$ of degree~$d$, one has:
\begin{enumerate}
\item[(i)] If $p=0$, then $\ind(X)$ divides $I_{d,N}$.
\item[(ii)] If $p>0$, the prime-to\nobreakdash-$p$ part of $\ind(X)$ divides $I_{d,N}$.
\item[(iii)] If $p>N$, then $\ind(X)$ divides $I_{d,N}$.
\end{enumerate}
\end{thm}

\begin{ex}
\label{ex:hyp}
(i) Let~$K$ denote either the field $\C((t))$, or the maximal unramified extension of a $p$\nobreakdash-adic field with $p\geq 5$.
Then any sextic hypersurface in $\P^3$ over~$K$ has a zero-cycle of degree~$1$,
and so does any hypersurface of degree~$12$ in $\P^4$.

(ii) If $d \leq N$, then $I_{d,N}=1$.  In this case, it is even true that $X(K) \neq \emptyset$, according to a theorem of Lang~\cite{Lang}.
\end{ex}

\begin{proof}[Proof of Theorem~\ref{thm:hypersurf}]
For $n \in \{1,\dots,N\}$, let $X_n$ denote the intersection of~$X$ with a linear subspace of~$\P^N$ of dimension~$n$.
Let $p \geq 0$ be the characteristic of~$k$.
According to Theorem~\ref{th:B},
the index of~$X$, if $p=0$, or its prime-to\nobreakdash-$p$ part if $p>0$, divides $\chi(X,\sO_{X_n})$ for every~$n$.
On the other hand, as~$X_n$ is a degree~$d$ hypersurface in~$\P^n_K$, we have $\chi(X,\sO_{X_n})=\chi_{d,n}$, where
\begin{align*}
\chi_{d,n}=1-(-1)^n \binom{d-1}{n}\rlap{\text{.}}
\end{align*}
The following lemma thus concludes the proof of~(i) and~(ii).

\begin{lem}
For any~$d \geq 1$ and~$N \geq 1$, we have
$I_{d,N} = \gcd \;\!\bigl\{ \chi_{d,n} \;\!; \; 1 \leq n \leq N \bigr\}$.
\end{lem}

\begin{proof}
We may rewrite $\chi_{d,n}$ as
\begin{align*}
\chi_{d,n}&= 1 - \prod_{1 \leq i \leq n} \bigg(1 - \frac{d}{\gcd(i,d)}\cdot\frac{1}{\frac{i}{\gcd(i,d)}}\bigg)\rlap{\text{.}}
\end{align*}
The integers $i/\gcd(i,d)$ and $d/\gcd(i,d)$ are coprime, and $I_{d,N}$ divides $d/\gcd(i,d)$ for $i \leq N$,
therefore $i/\gcd(i,d)$ and $I_{d,N}$ are coprime for $i\leq N$.  Thus, if $n\leq N$, each factor appearing in the above expression makes sense in $\Z/I_{d,N}\Z$.
As $I_{d,N}$ divides $d/\gcd(i,d)$ for $i \leq N$, we conclude that $I_{d,N}$ divides $\chi_{d,n}$ for all $n \leq N$.

It remains to be shown that $\gcd \;\!\!\bigl\{ \chi_{d,n} \;\!; \; 1 \leq n \leq N \bigr\}$ divides $I_{d,N}$.
For $N=1$ this is clear.  Assume $N\geq 2$ and
$\gcd \;\!\!\bigl\{ \chi_{d,n} \;\!; \; 1 \leq n \leq N-1 \bigr\}$ divides $I_{d,N-1}$.
If~$I_{d,N}=I_{d,N-1}$, there is nothing to prove.
Otherwise~$N$ must divide~$d$, and the desired result follows from the equality
\begin{align*}
1 - \chi_{d,N} = \big(1 - \chi_{d,N-1}\big)\bigg(1 - \frac{d}{N}\bigg)
\end{align*}
and from the induction hypothesis.
\end{proof}

Let us turn to~(iii).  Denoting by $v_p(n)$ the $p$\nobreakdash-adic valuation of an integer~$n$,
we remark that if $p>N$, then $v_p(I_{d,N})=v_p(d)$.
As~$X$ is a degree~$d$ hypersurface in projective space,
the index of~$X$ divides~$d$.  Therefore $v_p(\ind(X))\leq v_p(I_{d,N})$ if $p>N$.
In view of~(ii), this completes the proof of Theorem~\ref{thm:hypersurf}.
\end{proof}

Proposition~\ref{prop:twistedfermat} below shows that the bound given in Theorem~\ref{thm:hypersurf} is optimal,
as it is attained by ``maximally twisted'' Fermat hypersurfaces.

\begin{prop}
\label{prop:twistedfermat}
Let~$K$ be the quotient field of a discrete valuation ring~$R$,
and let $X_{d,N} \subset \P^N_K$ be the hypersurface defined by
\begin{align}
\label{eq:fermat}
x_0^d = \sum_{i=1}^N \pi^i x_i^d\rlap{\text{,}}
\end{align}
where~$\pi$ is a uniformiser of~$R$.
Then $\ind(X_{d,N})=I_{d,N}$ for any $d,N\geq 1$.
\end{prop}

\begin{proof}
For any $\delta \in \{1,\dots,N\}$ which divides~$d$, the $K(\pi^{\delta/d})$\nobreakdash-point with
coordinates $x_0=1$, $x_\delta=\pi^{-\delta/d}$, $x_i=0$ for $i \notin\{0,\delta\}$
lies on $X_{d,N}$.
Therefore $\ind(X_{d,N})$ divides~$I_{d,N}$.

To prove that $I_{d,N}$ divides $\ind(X_{d,N})$,
we may assume, by extending scalars, that~$K$ is complete.
Let us fix a closed point $x \in X$ and show that $I_{d,N}$ divides $e=\deg(x)$.
As~$K$ is complete, the integral closure~$R'$ of~$R$ in the residue field~$K'$ of~$x$
is a discrete valuation ring, and the corresponding valuation $v:K'^*\to \Z$
satisfies $v(\pi)=e$
(see \cite[Ch.~II, \textsection~2, Prop.~3 and Cor.~1]{SerreCL}).

Write the coordinates of~$x$ as $x=[x_0:\dots:x_N]$ where all~$x_i$'s belong to~$R'$ and one of them is a unit.
For $a=(a_0,\dots,a_N) \in \Z^{N+1}$, let us denote by~$m(a)$ the minimum of $ie+da_i$ over all $i \in \{0,\dots,N\}$.

\begin{lem}
\label{lem:constructb}
Assume $I_{d,N}$ does not divide~$e$. Then
for any $a \in \Z^{N+1}$ such that $v(x_i)\geq a_i$ for all~$i$,
there exists $b \in \Z^{N+1}$ such that $v(x_i) \geq b_i$ for all~$i$
and such that $m(b)>m(a)$.
\end{lem}

\begin{proof}
We first remark that the map $i \mapsto ie + da_i$ is injective on $\{0,\dots,N\}$.
Indeed, assume there exist $i,j$ with $0 \leq i<j \leq N$ such that $ie+da_i=je+da_j$.
Let $\delta=\gcd(j-i,d)$.  Then $\delta \in \{1,\dots,N\}$ and $d/\delta$ divides~$e$,
since $(j-i)e=d(a_i-a_j)$.  Thus $I_{d,N}$ divides~$e$, a contradiction.

In particular, there is a unique $i(a) \in \{0,\dots,N\}$ such that $i(a)e+da_{i(a)}=m(a)$.
Let $b_i=a_i$ for $i \neq i(a)$ and $b_{i(a)}=a_{i(a)}+1$.
It is clear that $m(b)>m(a)$.
Let us note, moreover, that $v(\pi^i x_i^d)>m(a)$ for $i \neq i(a)$.   Thanks to~\eqref{eq:fermat},
it follows that $v(\pi^i x_i^d)>m(a)$ for $i=i(a)$ as well.
In other words $i(a)e + dv(x_{i(a)}) > i(a)e + da_{i(a)}$, hence $v(x_{i(a)})>a_{i(a)}$.
As $v(x_i)\geq v(a_i)$ for all~$i$, we conclude that $v(x_i) \geq v(b_i)$ for all~$i$.
\end{proof}

Assume $I_{d,N}$ does not divide~$e$.  A repeated application of Lemma~\ref{lem:constructb},
starting with $0 \in \Z^{N+1}$, yields an $a \in \Z^{N+1}$
such that $v(x_i)\geq a_i$ for all~$i$ and $m(a)>Ne$.  The condition $m(a)>Ne$ implies $a_i>0$ for all~$i$,
which in turn contradicts the hypothesis that one of the~$x_i$'s is a unit.
\end{proof}

\begin{ex}
\label{ex:hypcont}
Let $d \geq 1$.  Write $d = \prod_{i=1}^n p_i^{\alpha_i}$ with pairwise distinct prime numbers~$p_i$.
Let $N_0 = \max(p_1^{\alpha_1},\dots,p_n^{\alpha_n})$.
Theorem~\ref{thm:hypersurf} and Proposition~\ref{prop:twistedfermat} show that the property
 ``any degree~$d$ hypersurface
in~$\P^N$ over $\C((t))$ possesses a zero-cycle of degree~$1$''
holds if and only if $N \geq N_0$.
\end{ex}

\smallskip\section{An extension to cobordism}
\label{sec:Cobordism}

The goal of the present section is to prove that in the situation of Theorem~\ref{thm:Divisibility}, not only does
the index of~$X$ over~$K$ (or its prime-to\nobreakdash-$p$ part) divide the Euler--Poincar\'e characteristic of any vector bundle on~$X$,
but it also divides the class of~$X$ in the cobordism ring of $\Spec(K)$
(Theorem~\ref{thm:Cobordism}).  Since~$K$ may have positive characteristic, we rely on an ad hoc definition
for the cobordism ring of $\Spec(K)$.  The definition we use is motivated by the fact that any
element of the complex cobordism ring is determined by the set of its Chern
numbers (see \cite[p.~117, Theorem]{Stong}).  As a consequence of Theorem~\ref{thm:Cobordism},
 integral-valued rational characteristic classes yield bounds on the index of smooth proper schemes
over the quotient field of a Henselian discrete valuation ring with algebraically closed residue field
(Corollary~\ref{cor:Cobordismcn}, Examples~\ref{ex:cobordismfirstexample}
to~\ref{ex:cobordismlastexample}).

We start by defining the cobordism ring of $\Spec(K)$.  Let $\Z[\mathbf b]=\Z[b_1,b_2,\dots]$ denote a polynomial ring in countably many variables.
Let~$\sI$ be the set of all sequences $(\alpha_j)_{j \geq 1}$ of nonnegative integers
all but finitely many of which are zero.
For $I \in \sI$, let $|I|=\sum j\alpha_j$ and $b^I = \prod b_j^{\alpha_j}$.
Given a smooth proper connected scheme~$X$ of dimension~$d$ over a field~$K$, we define, following Merkurjev~\cite{MerkurjevOrient}, the \emph{fundamental polynomial}
$b_K(X) \in \Z[\mathbf b]$ of~$X$ by the formula
\begin{align}
\label{eq:fundamentalpol}
b_K(X) = \sum_{|I|=d} \deg(c_I(-T_X)) b^I\rlap{\text{,}}
\end{align}
where
$c_I(-T_X) \in \CH_0(X)$ denotes the Conner--Floyd Chern class of the virtual vector bundle $-T_X \in K_0(X)$
associated to~$I$. (The Conner--Floyd Chern classes
are polynomials, with integer coefficients,
in the usual Chern classes; their definition is recalled below.)
For an arbitrary smooth proper $K$\nobreakdash-scheme~$X$, let~$b_K(X)$ be the sum of the fundamental polynomials of the connected components of~$X$.
The \emph{cobordism ring of $\Spec(K)$} is by definition the subring $\L_K \subseteq \Z[\mathbf b]$ generated by the
fundamental polynomials of all irreducible smooth proper schemes~$X$ over~$K$.

Even though we shall not use this property, let us remark that all elements of~$\L_K$ are in fact fundamental polynomials of smooth proper schemes over~$K$.
Indeed, the map $X \mapsto b_K(X)$ takes disjoint unions to sums, products to products
(a consequence of the Whitney sum formula for Conner--Floyd Chern classes, see \cite[Theorem~4.1]{Adams}),
and the set of fundamental polynomials of smooth proper $K$\nobreakdash-schemes
is stable under $b \mapsto -b$
according to Thom (see \cite[\textsection5]{Thom}).

We may now state the main result of this section.

\begin{thm}
\label{thm:Cobordism}
Let~$R$ be a Henselian discrete valuation ring with algebraically closed residue field~$k$ and quotient field~$K$.
Let~$X$ be a smooth proper irreducible $K$\nobreakdash-scheme.  Let~$p$ denote the characteristic of~$k$.
\begin{enumerate}
\item[(i)] If $p=0$, then $b_K(X)$ is divisible by $\ind(X)$ in the ring~$\L_K$.
\item[(ii)] If $p>0$, then $b_K(X)$ is divisible, in~$\L_K$, by the prime-to\nobreakdash-$p$ part of $\ind(X)$.
\item[(iii)] If $p>\dim(X)+1$, then $b_K(X)$ is divisible by $\ind(X)$ in~$\L_K$.
\end{enumerate}
\end{thm}

\begin{rem}
According to Quillen~\cite[Theorem 6.5]{Quillen}, the complex cobordism ring is canonically isomorphic to the
Lazard ring, which by definition is the coefficient ring of the universal rank one commutative formal group law.
Merkurjev~\cite[Theorem~8.2]{MerkurjevOrient} has shown that $\L_K$,
for any field~$K$, is also canonically isomorphic to the Lazard ring.
See our brief discussion of complex cobordism below for a more detailed
description of the relation between the Lazard ring, the complex cobordism
ring and the polynomial ring $\Z[\mathbf b]$.
\end{rem}

Before proving
Theorem~\ref{thm:Cobordism},
we set up some notation and state a few lemmas.
For the time being~$K$ denotes an arbitrary field.

Let $\Z[\mathbf c]=\Z[c_1,c_2,\dots]$ denote another
polynomial ring in countably many variables.
Let $I=(\alpha_j)_{j \geq 1} \in \sI$.
For $n \geq |I|$,
consider the monomial symmetric polynomial in~$n$ indeterminates
$$
\sigma_{I}=\sum \xi_1^{m_1}\cdots\xi_n^{m_n} \in \Z[\xi_1,\dots,\xi_n]\rlap{\text{,}}
$$
where the sum ranges over all $n$\nobreakdash-tuples of nonnegative integers $(m_1,\dots,m_n)$ such that for each $j \geq 1$,
exactly~$\alpha_j$ of the $m_i$'s are equal to~$j$. As~$\sigma_I$ is invariant under permutations of the $\xi_i$'s, there is a unique
$c_I \in \Z[\mathbf c]$ such that $\sigma_I=c_I(\sigma_1,\dots,\sigma_n)$, where $\sigma_i$ denotes the $i$th elementary symmetric polynomial in the $\xi_j$'s.
The polynomial $c_I \in \Z[\mathbf c]$ thus defined does not depend on the choice of $n \geq |I|$
(see \cite[Ch.~I, \textsection2]{Macdonald}).

We recall that the family $(c_I)_{I \in \sI}$ forms a basis of the $\Z$\nobreakdash-module $\Z[\mathbf c]$, and
that for $I \in \sI$, the Conner--Floyd Chern class $c_I(E)$ of a virtual vector bundle~$E$ on a smooth $K$\nobreakdash-scheme~$X$
is by definition the image of~$c_I$ by the ring homomorphism $\Z[\mathbf c] \to \CH^*(X)$, $c_i \mapsto c_i(E)$.

Let us consider $\Q[\mathbf b]$ and $\Q[\mathbf c]$ as graded algebras by letting $\deg(b_i)=i$ and $\deg(c_i)=i$.
The fundamental polynomial of an irreducible smooth proper scheme of dimension~$d$ over~$K$ is thus homogeneous of degree~$d$,
and for any $I \in \sI$, the polynomial $c_I$ is homogeneous of degree~$|I|$.
For $d \geq 0$, let $\Q[\mathbf b]_d$ and $\Q[\mathbf c]_d$ denote the homogeneous components of degree~$d$ of $\Q[\mathbf b]$ and $\Q[\mathbf c]$,
and let $\sI_d = \{I \in \sI; |I|=d\}$.
We define a perfect pairing of finite-dimensional $\Q$\nobreakdash-vector spaces
\begin{align}
\label{eq:pairing}
\langle \phantom{-}, \phantom{-}\rangle:\Q[\mathbf c]_d \times \Q[\mathbf b]_d \longrightarrow \Q
\end{align}
by declaring that the bases $(c_I)_{I \in \sI_d}$
and $(b^I)_{I \in \sI_d}$ are dual to each other.

We need to briefly recall a few facts about complex cobordism; we refer the reader to \cite[Chapter~1, \textsection\textsection1--4]{Adams}, \cite[Chapters~I--IV]{Stong}, and \cite[\textsection1]{Quillen} for details.
Let $\pi_*(MU)$ denote the complex cobordism ring, \emph{i.e.}, the generalized homology of the point with values in the Thom spectrum.
To any compact almost complex manifold~$M$ is associated an element~$[M]$ of the complex cobordism ring $\pi_*(MU)$.
These classes generate $\pi_*(MU)$ as an abelian group.
The Hurewicz map $\pi_*(MU) \to H_*(MU,\Z)$
and the Thom isomorphism
$H_*(MU,\Z)\simeq H_*(BU,\Z)$
yield a map $b:\pi_*(MU)\to H_*(BU,\Z)$, which is known to be injective.
If we endow $H_*(BU,\Z)$ with the ring structure induced by the direct sum map $\oplus:BU \times BU \to BU$, then~$b$ becomes a ring homomorphism.
Thus~$b$ identifies the complex cobordism ring $\pi_*(MU)$ with a subring of $\Z[\mathbf b]$; we shall denote this subring by $\L$.
For any compact almost complex manifold~$M$, the image in~$\Z[\mathbf b]$ of $[M] \in \pi_*(MU)$ is described by the defining formula of the fundamental polynomial~\eqref{eq:fundamentalpol}
(see \cite[(6.2), p.~49]{Quillen}).

 Let~$\Laz$ denote the Lazard ring, that is, the coefficient ring of the universal rank one commutative formal group law.
Letting $\lambda(t)\in\Z[\mathbf b][[t]]$ be the power series $t+\sum_{n\ge1}b_nt^{n+1}$ and $\lambda^{-1}(t)$ the inverse of $\lambda(t)$ with respect to composition of power series,  $\Laz$  embeds into $\Z[\mathbf b]$ via the classifying homomorphism associated to the formal group law $F(u,v)=\lambda(\lambda^{-1}(u)+\lambda^{-1}(v))$  with coefficients in $\Z[\mathbf b]$  (see for example \cite[II, \S 7]{Adams}).  According to Quillen~\cite[Theorem 6.5]{Quillen}, the image of~$\Laz$ in
$\Z[\mathbf b]$ is exactly the subring $\L$.   Finally, Merkurjev's result~\cite[Theorem~8.2]{MerkurjevOrient} states that $\L_K=\L$ for all fields $K$,
a fact that we shall reprove below using the Hattori--Stong theorem (Proposition~\ref{prop:CobordConclusion}).

According to Milnor (see, \emph{e.g.}, \cite[p.~86, Corollary 10.8]{Adams}),
the ring~$\L$ is generated by the classes of projective spaces $\P^n_\C$ for $n \geq 1$
and of the so-called \emph{Milnor hypersurfaces} $H_{m,n}$
with $2 \leq m \leq n$, where $H_{m,n}$
is the smooth hypersurface of bidegree $(1,1)$ in $\P^m_\C \times \P^n_\C$
defined by the equation $\sum_{i=0}^m x_i y_i=0$.

\begin{lem}
\label{lem:Cobord1}
We have the inclusion $\L \subseteq \L_K$ of subrings of $\Z[\mathbf b]$.
\end{lem}

\begin{proof}
If~$\sX$ is a smooth proper scheme over~$\Z$, the polynomial $b_K(\sX \otimes_\Z K)$ does not depend on the field~$K$.
As complex projective spaces and Milnor hypersurfaces possess smooth proper models over $\Spec(\Z)$,
the lemma follows.
\end{proof}

For any $d\geq 0$, let $\L_d = \L \cap \Q[\mathbf b]_d$ and $\L_{K,d}=\L_K \cap \Q[\mathbf b]_d$.
We recall that $\L_d$ is a lattice in $\Q[\mathbf b]_d$ (\emph{i.e.}, a finitely generated subgroup containing a basis),
see \cite[p.~117, Theorem]{Stong}.
As $\L \subseteq \L_K \subseteq \Z[\mathbf b]$,
the group $\L_{K,d}$ is also a lattice in $\Q[\mathbf b]_d$.
Let
$I_{K,d}=\big\{c \in \Q[\mathbf c]_d \; ; \; \langle c, \L_{K,d}\rangle \subseteq \Z\big\}$
denote the  lattice dual to $\L_{K,d}$
with respect to~\eqref{eq:pairing}.
Similarly, let $I_d \subset \Q[\mathbf c]_d$ denote the lattice dual to $\L_d$.

We need to introduce one more subgroup of $\Q[\mathbf c]_d$.
For any polynomial $f \in \Z[t_1,\dots,t_d]$ invariant under permutations of the $t_i$'s,
the homogeneous part of degree~$d$ of the power series
\begin{align}
f\big(e^{\xi_1},\dots,e^{\xi_d}\big)\prod_{j=1}^d\frac{\xi_j}{1-e^{-\xi_j}} \in \Q[[\xi_1,\dots,\xi_d]]
\end{align}
is symmetric in the $\xi_i$'s. Thus it can be written $R_f(\sigma_1,\dots,\sigma_d)$
for a unique
$R_f \in \Q[\mathbf c]_d$,
where~$\sigma_i$ denotes the $i$th elementary symmetric polynomial in the~$\xi_j$'s.
Let $s_1,s_2,\ldots \in \Z[\mathbf c]$ be the Segre polynomials, defined by the formula
\begin{align}
\label{eq:Segre}
1 + \sum_{i\geq 1} s_i t^i = \Big(1 + \sum_{i\geq 1} c_i t^i\Big)^{-1}  \in \Z[\mathbf c][[t]] \rlap{\text{.}}
\end{align}
Let $S_f = R_f(s_1,\dots,s_d) \in \Q[\mathbf c]_d$.
Let $I'_d \subset \Q[\mathbf c]_d$ be the subgroup consisting of the polynomials~$S_f$
when~$f$ ranges over all symmetric polynomials in $\Z[t_1,\dots,t_d]$.

Finally,
let $I_K = \bigoplus_{d \geq 0} I_{K,d}$, $I= \bigoplus_{d\geq 0}I_d$, and $I'=\bigoplus_{d \geq 0} I'_d$.
The definition of~$I'$ is motivated by the following lemma.

\begin{lem}
\label{lem:CobordSf}
Let~$X$ be a smooth proper irreducible variety over a field~$K$,
of dimension~$d$.
Let $m_1,\dots,m_d$ be nonnegative integers, and let
$f=\prod \tau_i^{m_i}$, where~$\tau_i$ denotes the $i$th elementary symmetric polynomial in $t_1,\dots,t_d$.
Then
\begin{align*}
\big\langle S_f, b_K(X)\big\rangle = \chi\Big(X,\;\bigotimes_{i=1}^d \Big(\bigwedge^i T_X\Big)^{\otimes m_i}\Big)\rlap{\text{.}}
\end{align*}
\end{lem}

\begin{proof}
The number
$\langle S_f, b_K(X)\rangle$ is the degree of the element of $\CH_0(X) \otimes_\Z \Q$
obtained by evaluating the polynomial $R_f \in \Q[\mathbf c]$
on the Chern classes of~$T_X$.
In addition,
if $\xi_1,\dots,\xi_d$ denote the Chern roots of $T_X$,
we have
\begin{align*}
\ch\Big(\bigotimes_{i=1}^d \big(\bigwedge^i T_X\big)^{\otimes m_i}\Big)=
\prod \ch\big(\bigwedge^i T_X\big)^{m_i} = f(e^{\xi_1},\dots,e^{\xi_d})\rlap{\text{.}}
\end{align*}
Hence the lemma results from
the Hirzebruch--Riemann--Roch theorem.
\end{proof}

As~$S_f$ depends linearly on~$f$ and as any symmetric polynomial $f\in\Z[t_1,\dots,t_d]$ is a $\Z$\nobreakdash-linear combination of monomials $\prod \tau_i^{m_i}$,
Lemma~\ref{lem:CobordSf}
gives the value of $\langle S_f, b_K(X)\rangle$ for any $S_f \in I'_d$.
In particular, it implies that $I' \subseteq I_K$;
thus, according to Lemma~\ref{lem:Cobord1}, we have $I' \subseteq I_K \subseteq I$.
On the other hand, a theorem due to Hattori and Stong asserts that $I'=I$ (see \cite[Theorem 1]{StongRR}, \cite[Theorem I]{Hattori}).
As a result,
letting $\L'_d =\big\{b \in \Q[\mathbf b]_d \; ; \; \langle I'_d\mkern1mu, b\rangle \subseteq \Z\big\}$ and $\L'=\bigoplus_{d\geq 0} \L'_d$,
we have established the following proposition:

\begin{prop}
\label{prop:CobordConclusion}
For any field~$K$, we have $I'=I_K=I$ and $\L=\L_K=\L'$.
\end{prop}

We are now in a position to complete the proof of Theorem~\ref{thm:Cobordism}.

\begin{proof}[Proof of Theorem~\ref{thm:Cobordism}]
Let~$X$ and~$K$ be as in the statement of the theorem.
Let $n=\ind(X)$ if $p=0$ or $p>\dim(X)+1$, otherwise let~$n$ denote the prime-to\nobreakdash-$p$ part of $\ind(X)$.
According to Theorem~\ref{thm:Divisibility} and Lemma~\ref{lem:CobordSf}, we have $\langle I', b_K(X) \rangle \subseteq n\Z$.
In other words, the class $b_K(X)$ is divisible by~$n$ in $\L'$.  Thanks to Proposition~\ref{prop:CobordConclusion}, we deduce
that it is divisible by~$n$ in $\L_K$.
\end{proof}

Corollary~\ref{cor:Cobordismcn} is an
essentially equivalent reformulation 
of Theorem~\ref{thm:Cobordism}
in terms of characteristic numbers.
To ease notation we write $P(T_X)$ for $P(c_1(T_X),\dots,c_d(T_X))$.

\begin{cor}
\label{cor:Cobordismcn}
Let~$R$ be a Henselian discrete valuation ring with algebraically closed residue field~$k$ and quotient field~$K$.
Let~$p$ denote the characteristic of~$k$.
Let $d \geq 1$ and let $P \in \Q[c_1,\dots,c_d]$ be homogeneous of degree~$d$ with respect to the grading $\deg(c_i)=i$.
Assume that
 $\deg(P(T_X))\in \Z$ for any $d$\nobreakdash-dimensional product~$X$ of complex projective spaces and Milnor hypersurfaces.
Then, for any smooth proper irreducible $K$\nobreakdash-scheme $X$ of dimension~$d$, the rational number $\deg(P(T_X))$ is an integer, and we have:
\begin{enumerate}
\item[(i)] If $p=0$, then $\ind(X)$ divides $\deg(P(T_X))$.
\item[(ii)] If $p>0$, then the prime-to\nobreakdash-$p$ part of $\ind(X)$ divides $\deg(P(T_X))$.
\item[(iii)] If $p>\dim(X)+1$ or if the denominators of the coefficients of~$P$ are prime to~$p$, then $\ind(X)$ divides $\deg(P(T_X))$.
\end{enumerate}
\end{cor}

\begin{proof}
Let $Q=P(s_1,\dots,s_d) \in \Q[\mathbf c]_d$, where $s_1,\dots,s_d$ are the Segre polynomials (see~\eqref{eq:Segre}).
The linear form $\Q[\mathbf b]_d \to \Q$, $b \mapsto \langle Q, b\rangle$ maps $b_K(X)$ to $\deg(P(T_X))$, and is integral-valued on~$\L_d$
since~$\L_d$ is spanned by the fundamental polynomials of $d$\nobreakdash-dimensional products of complex projective spaces and Milnor hypersurfaces.
Thus,
by Proposition~\ref{prop:CobordConclusion}, it restricts to a homomorphism $\L_{K,d} \to \Z$.
Applying Theorem~\ref{thm:Cobordism} now yields the desired result, noting, for the third assertion, that if~$n$ is the lowest common multiple
of the denominators of the coefficients of~$P$,
then the index of~$X$ divides $n\deg(P(T_X))$ since $nP(T_X) \in \CH_0(X)$.
\end{proof}

\begin{rem}
\label{rem:virtualrep}
If $P \in \Q[\mathbf c]$ satisfies the hypothesis of Corollary~\ref{cor:Cobordismcn},
Proposition~\ref{prop:CobordConclusion} implies that for any field~$K$ and any smooth proper irreducible $K$\nobreakdash-scheme~$X$
of dimension~$d$, the rational number $\deg(P(T_X))$ may be written, in a way which does not depend on~$X$,
as a $\Z$\nobreakdash-linear combination of Euler--Poincar\'e characteristics of tensor products of exterior powers of~$T_X$.
In other words, there exists a virtual representation~$\rho$ of $\GL_{d,K}$ such that $\deg(P(T_X))=\chi(X,\rho(T_X))$.

More precisely, by a theorem of Serre~\cite[\textsection3.6, Th\'eor\`eme 4]{SerreGpRep}, the Grothendieck group of the category of representations of $\GL_{d,K}$
over~$K$ does not depend on the field~$K$: sending a representation~$\rho$ to its character
(viewed as a symmetric function in the characters $t_1,\dots,t_d$ of a maximal torus $\G_\mathrm{m}^d \subset \GL_d$) induces an isomorphism
\begin{align}
K_0(\Rep_K(\GL_{d,K}))\isoto\Z[t_1,\dots,t_d]^{\mathfrak S_d}\Big[\frac{1}{t_1\cdots t_d}\Big]\rlap{\text{.}}
\end{align}
Thus, writing~$P$ as $R_f$ for an $f \in \Z[t_1,\dots,t_d]^{\mathfrak S_d}$
and letting~$\rho$ be the
 unique virtual representation of $\GL_{d,K}$ with character~$f$,
we have
$\deg(P(T_X))=\chi(X,\rho(T_X))$
for any field~$K$ and any smooth proper irreducible $K$\nobreakdash-scheme~$X$ of
dimension~$d$.
\end{rem}

\medskip
The remainder of this section is devoted to examples.  From now on, the letter~$K$ will always denote the quotient
field of a Henselian discrete valuation ring with algebraically closed residue field of characteristic $p \geq 0$.

\begin{ex}
\label{ex:cobordismfirstexample}
The polynomial $P=\frac{1}{2}{c_d}$ satisfies the hypothesis of Corollary~\ref{cor:Cobordismcn} if~$d$ is odd.
Indeed, by Poincar\'e duality, the topological Euler--Poincar\'e characteristic of any odd-dimensional compact complex manifold is even.

As a consequence, if $p\neq 2$, the index of
any odd-dimensional smooth proper irreducible $K$\nobreakdash-scheme~$X$
divides $\frac{1}{2}e(X)$.
Here $e(X)$ denotes the $\ell$\nobreakdash-adic Euler--Poincar\'e characteristic of $X\otimes_K \bar K$
for any $\ell$ invertible in~$K$, where $\bar K$ is a separable closure of~$K$ (see \cite[Corollaire~4.9]{JouanolouSGA5}).

In this example, it is easy to exhibit the virtual representation~$\rho$ of $\GL_d$ whose existence is predicted by
Remark~\ref{rem:virtualrep}.  Namely, if~$V$ denotes the standard representation of $\GL_d$
and~$V^*$ is its dual,
the virtual representation
\begin{align}
\rho(V) = \sum_{i=0}^{\frac{d-1}{2}} (-1)^i \bigwedge^i V^*
\end{align}
satisfies $\frac{1}{2}e(X) = \chi(X,\rho(T_X))$.  Indeed, we have
$e(X)=\sum_{i=0}^d (-1)^i \chi(X,\Omega^i_{X/K})$
according to \cite[Proposition~4.11]{JouanolouSGA5},
and
$(-1)^i\chi(X,\Omega^i_{X/K})=(-1)^{d-i}\chi(X,\Omega^{d-i}_{X/K})$
for all~$i$ by Serre duality.
\end{ex}

\begin{ex}
The polynomial $P=\frac{1}{2}{c_1^d}$ satisfies the hypothesis of Corollary~\ref{cor:Cobordismcn} if~$d$ is odd.
To see this, first note that if $X=A \times B$ with $\dim(A)=a$, $\dim(B)=b$, then
$$
\deg\big(c_1^{a+b}(T_X)\big) = \binom{a+b}{a} \deg\big(c_1^a(T_A)\big) \deg\big(c_1^b(T_B)\big)\rlap{\text{;}}
$$
thus it suffices to check that~$P$ is integral-valued on
odd-dimensional projective spaces and Milnor hypersurfaces.
We have $\frac{1}{2}\deg\big(c_1^d(T_{\P^d})\big)=\frac{1}{2}(-1)^d(d+1)^d$, which is an
integer when~$d$ is odd, and
if  $d=m+n-1$ with $2 \leq m \leq n$,
the adjunction formula shows that
\begin{align*}
\frac{1}{2}\deg\big(c_1^d(H_{m,n})\big)=(-1)^d  \binom{d-1}{m-1} m^{m-1} n^{n-1} d\rlap{\text{,}}
\end{align*}
which is an integer as well.

Hence, if $p\neq 2$, the index of any smooth proper irreducible $K$\nobreakdash-scheme~$X$ of odd dimension~$d$ divides $\frac{1}{2}(K_X)^d$.
In contrast with the previous example,
we were unable in this case to find a closed-form formula (valid for all odd~$d$) for a virtual representation~$\rho_d$ of $\GL_d$ such that
$\frac{1}{2}(K_X)^d = \chi(X,\rho_d(T_X))$ for all $d$\nobreakdash-dimensional~$X$.
Such virtual representations do exist as a consequence of the Hattori--Stong theorem (see
Remark~\ref{rem:virtualrep}).
\end{ex}

\begin{ex}
Let~$d$ be an even integer.
For any compact complex manifold~$X$ of dimension~$d$,
the symmetric bilinear form $H^d(X,\R) \times H^d(X,\R) \to \R$ given by cup-product is symmetric, and hence has a well-defined signature $\sigma(X)$.
Hirzebruch's signature formula \cite[Theorem~8.2.2]{Hirzebruch}
furnishes a homogeneous degree~$d$ polynomial $P_d \in \Q[\mathbf c]$
such that $\sigma(X)=\deg(P_d(T_X))$ for all such~$X$.
Specifically, to obtain~$P_d$, evaluate the polynomial
denoted $L_{d/2}$ in \cite[\textsection1.5]{Hirzebruch}
at $p_i=\sum_{j=0}^{2i} (-1)^{i+j}c_j c_{2i-j}$, with the convention that $c_0=1$.
The first values of~$P_d$ are $P_2=\frac{1}{3}c_1^2 - \frac{2}{3}c_2$,
$P_4=\frac{1}{45}\big(14c_4-14c_1c_3+3c_2^2+4c_2c_1^2-c_1^4)$, etc.

For any field~$k$ and any smooth proper irreducible $k$\nobreakdash-scheme~$X$ of even dimension, we may define the ``signature'' of~$X$
as $\sig(X)=\deg(P_{\dim(X)}(T_X))$.
In case~$k$ admits an embedding $\tau:k\hookrightarrow \C$, we have $\sig(X)=\sigma(X(\C))$; in particular, $\sig(X)$ is an integer and $\sigma(X(\C))$ is independent of the choice of~$\tau$.
Thanks to Proposition~\ref{prop:CobordConclusion}, we conclude that
$\sig(X)$ is always an integer, even when~$k$ has positive characteristic
and~$\sig(X)$ has no interpretation as the signature of a real quadratic form.

By Corollary~\ref{cor:Cobordismcn}, if~$X$ is a
smooth proper irreducible $K$\nobreakdash-scheme of even dimension, the index of~$X$ divides $\sig(X)$ if $p=0$ or $p>\dim(X)+1$,
and the prime-to\nobreakdash-$p$ part of $\ind(X)$ divides $\sig(X)$ otherwise.

A representation~$\rho$ of $\GL_d$ such that $\sig(X) = \chi(X,\rho(T_X))$ for any field~$k$ and any
smooth proper irreducible $k$\nobreakdash-scheme~$X$ of even dimension~$d$
is given by
\begin{align}
\rho(V) = \bigoplus_{i=0}^{d} \bigwedge^i V^*\rlap{\text{,}}
\end{align}
where~$V$ denotes the standard representation.   Indeed, when $k=\C$, Hodge has proved that $\sigma(X(\C))=\sum_{i=0}^d \chi(X,\Omega^i_{X/\C})$
(see~\cite{Hodge} and~\cite[Introduction~0.6]{Hirzebruch}),
from which it follows
that $\sig(X)=\sum_{i=0}^d \chi(X,\Omega^i_{X/k})$ for any field~$k$.
\end{ex}

\begin{ex}
\label{ex:steenrodexample}
Let~$q$ be a prime number,
let $I=(\alpha_j)_{j \geq 1} \in \sI$, and let $d=|I|$.
Assume that $\alpha_j=0$ for every~$j$ which is not of the form $q^n-1$ for some $n \geq 1$,
and denote by $s_1,\dots,s_d \in \Z[\mathbf c]$ the Segre polynomials (see~\eqref{eq:Segre}).
Then
\begin{align*}
P=\frac{1}{q}c_I(s_1,\dots,s_d) \in \Q[\mathbf c]_d
\end{align*}
satisfies the hypothesis of Corollary~\ref{cor:Cobordismcn}.
Indeed, for any irreducible smooth projective scheme~$X$ of dimension~$d$ over a field of characteristic $\neq q$, the integer $\deg(c_I(-T_X))$ is divisible by~$q$, as follows from
\cite[Proposition 5.3]{MerkurjevSteenrodOps} applied to the structure morphism of~$X$.

Thanks to Corollary~\ref{cor:Cobordismcn}, we conclude that if $p \neq q$,
the index of any irreducible smooth proper $K$\nobreakdash-scheme of
dimension~$d$ divides $\frac{1}{q}\deg(c_I(-T_X))$.
\end{ex}

\begin{ex}
\label{ex:halfsegre}
Taking $q=2$ and $I=(d,0,\dots)$ in Example~\ref{ex:steenrodexample}, we see that
for any integer~$d$,
the polynomial $P=\frac{1}{2}s_d$
satisfies the hypothesis of Corollary~\ref{cor:Cobordismcn}.
\end{ex}

\begin{ex}
\label{ex:cobordismlastexample}
Let $I=(\alpha_j)_{j \geq 1}$ with $\alpha_d=1$ and $\alpha_j=0$ for $j \neq d$.
In this case, the polynomial~$c_I$ is the $d$th Newton polynomial $Q_d$
(see \cite[Part~II, \textsection12]{Adams}).
According to Example~\ref{ex:steenrodexample}, for any irreducible smooth proper $K$\nobreakdash-scheme~$X$ whose dimension is of the form $d=q^n-1$ for some prime $q \neq p$ and
some $n\geq 1$, the index of~$X$ divides $\frac{1}{q}\deg(Q_d(-T_X))$.

This example and Example~\ref{ex:halfsegre} may be combined as follows:
fix integers $n,m \geq 0$ and a prime $q\neq p$, and let $d=m(q^n-1)$.
Take $I=(\alpha_j)_{j \geq 1}$ with $\alpha_{q^n-1}=m$ and $\alpha_j=0$ for $j\neq q^n-1$. Then~$c_I(E)$,
for any vector bundle~$E$,
is the $m$th elementary symmetric function in the $(q^n-1)$st powers of the Chern roots of~$E$. For any irreducible smooth
proper $K$\nobreakdash-scheme~$X$ of dimension~$d$, the index of~$X$
divides $\frac{1}{q}\deg(c_I(-T_X))$.
\end{ex}

\bibliographystyle{amsplain}
\bibliography{indchi}

\end{document}